\documentclass[12pt,oneside,reqno]{amsart}

\usepackage{amsmath}
\usepackage{amsthm,amssymb,mathrsfs,mathtools,color}
\usepackage{dsfont}

\usepackage[pagebackref,hypertexnames=false]{hyperref}

\newcommand{\C}{\mathbb{C}}
\newcommand{\E}{\mathbb{E}}
\newcommand{\F}{\mathbb{F}}
\newcommand{\Fq}{\F_q}
\newcommand{\Pq}{\mathcal{P}_q}
\newcommand{\ov}{\overline}

\newcommand{\floor}[1]{\left\lfloor{#1}\right\rfloor}

\newcommand{\gen}[1]{\left\langle #1 \right\rangle}

\newcommand{\GH}{\widehat{G/H}}

\DeclareMathOperator{\ER}{ER}\DeclareMathOperator{\RC}{RC}

\newcommand{\ol}{\overline}
\newcommand{\ul}{\underline}

\newcommand{\bz}{\mathbf{0}}

\newcommand{\Sq}{\mathcal{S}_q}
\newcommand{\Sqr}{\Gamma^{(r)}}
\newcommand{\Sqrd}{\widehat{\Fq^\times/\Sqr}}

\newcommand{\bk}{\mathbf{k}}
\newcommand{\bt}{\mathbf{t}}

\DeclareMathOperator{\Cay}{Cay}

\newcommand{\vcdim}{\operatorname{VCdim}}
\newcommand{\vcdimg}{\vcdim_G}
\newcommand{\vcdimgr}{\vcdimg'}
\newcommand{\vcdimf}{\vcdim_{\Fq}}

\newcommand{\aqr}{\alpha_q^{(r)}}
\newcommand{\bqr}{\beta_q^{(r)}}
\newcommand{\mqr}{m_q^{(r)}}
\newcommand{\uar}{\ol{\alpha}^{(r)}}
\newcommand{\lar}{\ul{\alpha}^{(r)}}
\newcommand{\ubr}{\ol{\beta}^{(r)}}
\newcommand{\lbr}{\ul{\beta}^{(r)}}

\newcommand{\geqs}{\geqslant}
\newcommand{\leqs}{\leqslant}

\newcommand{\1}{\mathds{1}}

\newtheorem{theorem}{Theorem}[section]
\newtheorem{lemma}[theorem]{Lemma}

\newtheorem{conjecture}[theorem]{Conjecture}
\newtheorem{proposition}[theorem]{Proposition}

\theoremstyle{definition}

\newtheorem{definition}[theorem]{Definition}

\begin{document}
 
\title{The VC dimension of quadratic residues in finite fields}

\author{Brian McDonald}
\address{Department of Mathematics, University of Georgia}
\email{\href{mailto:brian.mcdonald@uga.edu}{brian.mcdonald@uga.edu}}

\author{Anurag Sahay}
\address{Department of Mathematics, Purdue University}
\email{\href{mailto:anuragsahay@purdue.edu}{anuragsahay@purdue.edu}}

\author{Emmett L. Wyman}
\address{Department of Mathematics and Statistics, Binghamton University}
\email{\href{mailto: ewyman@math.binghamton.edu}{ewyman@math.binghamton.edu}}

\begin{abstract}

We study the Vapnik--Chervonenkis (VC) dimension of the set of quadratic residues (i.e. squares) in finite fields, $\Fq$, when considered as a subset of the additive group. We conjecture that as $q \to \infty$, the squares have the maximum possible VC-dimension, viz. $(1+o(1))\log_2 q$. We prove, using the Weil bound for multiplicative character sums, that the VC-dimension is $\geqs (\tfrac{1}{2} + o(1))\log_2 q$. We also provide numerical evidence for our conjectures. The results generalize to multiplicative subgroups $\Gamma \subseteq \Fq^\times$ of bounded index.

\end{abstract}

\maketitle

\section{Introduction}

Consider the setting of an Abelian group written additively. For any subset, one defines the Vapnik--Chervonenkis dimension (VC-dimension for conciseness) of $S$ as follows.

\begin{definition}[VC-dimension in groups]\label{def: vcdimg} Let $(G,+)$ be an Abelian group, and let $S \subseteq G$ be an arbitrary subset. We say that $S$ \emph{shatters} $Y = \{y_1,\cdots,y_n\} \subseteq G$ if for any $\emptyset \subseteq A \subseteq Y$, there is an $x\in G$ such that 
\[ y_j \in A \iff y_j \in (S+x)\]
for $1\leqs j\leqs n$. Here, and throughout, we assume that $|Y| = n$ so that $y_j \neq y_k$ for $1 \leqs j < k \leqs n$. Further, we define the VC-dimension of $S$ to be the cardinality of the largest set $Y$ that is shattered by $S$, and we denote it by $\vcdimg(S)$
\end{definition}

Readers familiar with the definition of VC-dimension in the language of set systems or classifiers (see, e.g., \cite[Chapter~6]{understandingML}) should note that this is the usual notion of VC-dimension relative to the set system $\mathscr{P}_S$ of translates of $S$, 
\[ \mathscr{P}_S = \{ S + x : x \in G\},\]
where $S + x = S + \{x\} = \{ z + x : z \in S\}$, or equivalently relative to the classifiers 
\[ \mathscr{H}_S = \{ \1_{S+x} : x\in G, \} \]
where here and throughout, $\1_S:G \to \{0,1\}$ is the indicator function of the set $S$. Further, a standard argument shows that $\vcdimg(S) \leqs \log_2 |G|$ for every $S$, where $\lvert \cdot \rvert$ denotes cardinality.

A class of ambient groups of interest to us are the vector spaces over finite fields -- i.e., $G = \Fq^d$ for prime power $q$ and some dimension $d\geqs 1$. A specific instance of this type was considered by the Fitzpatrick, Iosevich, McDonald, and Wyman \cite{circlesFIMW}, where they investigated a circle in a plane $G = \Fq^2$. Another example of this type was considered by Iosevich, McDonald, and Sun \cite{dotproductIMS} where they considered $G = \Fq^3$, and sets defined using dot products. Since the initial dissemination of our work, several related works along similar themes as \cite{circlesFIMW,dotproductIMS} have appeared \cite{REUpaper,REUpaper2,phamparallel,refpaper}.
 
Our objective in this paper is to investigate the VC-dimension of sets with multiplicative structure in $\Fq$, when considered as subsets of the additive group $(\Fq,+)$. Our particular interest is in the set of quadratic residues (i.e. squares) $\Sq \subseteq \Fq$ given by
\[ \Sq = \{ x^2 : x \in \Fq^\times\}. \]
We remark at this point that our analysis can be easily adjusted to include the number $0$ in $\Sq$ if the reader desires. For economy of notation, however, we adopt the point of view that $0$ is ``half" an element of $\Sq$, a decision which we shall explain later. Concretely, we mean that
\begin{equation*}
\1_{\Sq}(y) = \begin{cases}
1 & \text{ if } y = x^2, x \in \Fq^\times, \\
1/2 & \text{ if } y = 0,\\
0 & \text{ otherwise,}
\end{cases}
\end{equation*}
where $\1_{\Sq}$ is the indicator function of $\Sq$. This requires some modification to the meaning of \emph{shattering} from Definition~\ref{def: vcdimg}; for our purposes, it suffices to add the extra requirement that the $x \in G$ such that $A = Y \cap (\Sq + x)$ satisfies $x \notin Y$.

Another viewpoint on our problem is provided by using graph theoretic terminology. In this article, graphs are directed, do not have multiple edges, and may have loops. We denote graphs also by the letter $G$, but this should cause no confusion, as the the underlying vertex set of our graphs shall be Abelian groups. For any vertex $v \in G$, the (out)-neighbourhood of $v$ is given by
\[ N(v) = \{ w \in G : v \to w \text{ in } G \},\] 
and the VC-dimension of the set family of neighbourhoods is called the VC-dimension of the graph. This is encapsulated in the following definition.

\begin{definition}[VC-dimension in graphs] Let $G = (V,E)$ be a directed graph. We say that $E$ shatters $Y = \{y_1,\cdots,y_n\} \subseteq G$ if for any $\emptyset \subseteq A \subseteq Y$, there is an $x\in G$ such that 
\[ y_j \in A \iff x \to y_j \text{ in } G \]
for $1\leqs j\leqs n$. Further, we define the VC-dimension of $G = (V,E)$ to the be cardinality of the largest set $Y$ that is shattered by $E$, and we denote it by $\vcdimgr(E)$
\end{definition}
The astute reader will notice that the above generalizes Definition~\ref{def: vcdimg} as for a group $G$ and a subset $S \subseteq G$, if we set $E$ to be the edge-set of the Cayley (di)graph $\Cay(G,S)$, then
\[ \vcdimg(S) = \vcdimgr(E).\]
Thus, if we define $\Pq = \Cay(\Fq,\Sq)$ to be the usual notion of the Paley (di)graph, we see that the problem we are interested in is exactly the question of the VC-dimension of $\Pq$. Paley graphs are an important example of pseudorandom graphs (see, for example, the seminal work of Chung, Graham, and Wilson \cite{pseudorandomgraphsCGW}).

VC-dimension can be thought of as a measure of complexity when viewed internally from a structure (e.g., the group structure or the graph structure). Thus, if a set $S \subseteq G$ has high additive structure, one might expect a low VC-dimension. Conversely, if $S$ has low additive structure, one should expect a high VC-dimension. There are numerous recent works in the combinatorics literature where the assumption of bounded VC-dimension delivers strong structure on the involved sets (see, for example, \cite{alonfoxzhao},\cite{sisask},\cite{foxpachsuk1},\cite{foxpachsuk2}, \cite{lovaszegedy}, \cite{alonfischernewman}). 

Since $\Sq$ is clearly multiplicatively structured, and since one expects that addition and multiplication in $\Fq$ are only loosely correlated due to the sum-product phenomenon (see, for example, \cite{sumproductBKT},\cite{sumproductT},\cite{sumproductRSS}, and references therein), it is natural to expect that $\Sq$ will have the maximum possible VC-dimension, $\log_2 q$. Furthermore, since $\Pq$ is pseudorandom, it is similarly natural to expect that it would have maximum possible VC-dimension.

To state our expectations in this regard, we introduce two quantities $\alpha_q$ and $\beta_q$. The former is defined by
\[ \alpha_q = \frac{\vcdimf(\Sq)}{\log_2 q}. \]
For the latter, we define first $m_q$ to be the largest integer such that \emph{all} subsets $Y \subseteq \Fq$ of size $|Y| \leqs m_q$ are shattered by $\Sq$. Then,
\[ \beta_q = \frac{m_q}{\log_2 q}. \] 
The quantity $m_q$ is called the testing dimension by some authors (for example \cite{VCdimrandomgraph} and \cite{testingdimension}). It is straightforward to see that $0\leqs \beta_q \leqs \alpha_q\leqs 1$. Finally, we further define
\begin{equation} \label{ab def} \begin{split} 
\ol{\alpha} = \limsup_{q\to\infty} \alpha_q, &\qquad \ul{\alpha} = \liminf_{q\to\infty} \alpha_q,\\
\ol{\beta} = \limsup_{q\to\infty} \beta_q, &\qquad \ul{\beta} = \liminf_{q\to\infty} \beta_q. \end{split}\end{equation}
We can now formally state our aforementioned expectation. 
\begin{conjecture}\label{conj: main squares} Let $\ul{\alpha}$ be as defined in \eqref{ab def}. Then, $\ul{\alpha} = 1$. In particular, it follows that $\ol{\alpha} = \ul{\alpha}$ and further, that
\[ \vcdimf(\Sq) = (1+o(1))\log_2 q. \]
\end{conjecture}
We are not able to prove this conjecture. However, we have partial progress, in the form of the following theorem.
\begin{theorem}\label{thm: main squares} Let $\ul{\beta}$ be as defined in \eqref{ab def}. Then, $\ul{\beta} \geqs 1/2$. In particular, for any $\epsilon > 0$, we have that for all $q \gg_\epsilon 1$, and for every subset $Y \subseteq \Fq$ with $|Y| \leqs (\tfrac{1}{2}-\epsilon) \log_2 q$, the set of squares $\Sq$ shatters $Y$. 
\end{theorem}
In particular, it is immediate from this theorem that $\ul{\alpha} \geqs 1/2$, and further that
\[ \vcdimf(\Sq) \geqs (\tfrac{1}{2}-\epsilon)\log_2 q. \]
for all $q\gg_\epsilon 1$.

Our results in this regard follow from the broader context where $\Sq$ (which is a multiplicative subgroup of index $2$) is replaced by a more general multiplicative subgroup $\Gamma \subseteq \Fq^\times$. In particular, recall that since $\Fq^\times$ is cyclic, for every $r \mid (q-1)$, $\Fq^\times$ has exactly one subgroup of index $r$, namely the set of $r$th powers, 
\[ \Sqr = \{ x^r : x \in \Fq^\times\}, \]
so that $\Gamma^{(2)} = \Sq$. Further, as with squares, we adopt the point of view that $0$ is a ``fuzzy" member of $\Sqr$ with weight $1/r$. In fact, we assume that for each multiplicative coset $t\Sqr = \{ tx^r : x \in \Fq^\times\}$, $0$ is a fuzzy member of $t\Sqr$ with weight $1/r$. That is, for $t\neq 0$,
\begin{equation} \label{fuzzy}
\1_{t\Sqr}(y) = \begin{cases}
1 & \text{ if } y  \in t\Sqr, \\
1/r & \text{ if } y = 0,\\
0 & \text{ otherwise.}
\end{cases}
\end{equation}

Analogous to the $r = 2$ case, we define
\[ \aqr = \frac{\vcdimf(\Sqr)}{\log_2 q},\qquad \bqr = \frac{\mqr}{\log_2 q}, \]
where now $\mqr$ is the largest integer such that all subsets $Y \subseteq \Fq$ of size $|Y| \leqs \mqr$ are shattered by $\Sqr$. We further define
\begin{equation} \label{ar br def} \begin{split} 
\uar = \limsup_{q\to\infty} \aqr, &\qquad \lar = \liminf_{q\to\infty} \aqr,\\
\ubr = \limsup_{q\to\infty} \bqr, &\qquad \lbr = \liminf_{q\to\infty} \bqr. \end{split}\end{equation}
Then, provided that $\Sqr$ is of bounded index (i.e., $r\ll 1$), we expect the natural generalization of Conjecture~\ref{conj: main squares} to hold, and are able to able to prove the analogous partial progress, with $1/2$ replaced by $(\log_r 2)/2$
\begin{conjecture} \label{conj: main gen} Let $\lar$ be as defined in \eqref{ar br def}. Then, $\lar = 1$. In particular, it follows that $\uar = \lar$ and further that
\[ \vcdimf(\Sqr) = (1+o_r(1))\log_2 q. \]
\end{conjecture}
\begin{theorem} \label{thm: main gen} Let $\lbr$ be as defined in \eqref{ar br def}. Then, $\lbr \geqs (\log_r 2)/2$. In particular, for any $\epsilon > 0$, we have that for all $q \gg_{r,\epsilon} 1$, and for every subset $Y \subseteq \Fq$ with $|Y| \leqs (\tfrac{1}{2}-\epsilon) \log_r q$, the set of $r$th powers $\Sqr$ shatters $Y$. 
\end{theorem}

It is clear that Theorem~\ref{thm: main squares} is a corollary of this one. We now focus on this theorem, and provide an overview of its proof. Broadly speaking, we proceed by showing that $\Sqr$ shatters $Y = \{y_1,\cdots,y_n\}$ provided a certain expression involving the indicator functions of the multiplicative cosets $t\Sqr$ evaluated at $y_j - x$ is positive. Then, we Fourier-expand these indicator functions using the (multiplicative) characters on the quotient group $\Fq^\times/\Sqr$. Since $\Fq^\times$ is cyclic, so is $\Sqrd$, and hence everything can be written in terms of its generator $\chi_r$. Upon doing so, the task reduces to finding nontrivial bounds for sums of the form
\[ \sum_{x\in \Fq} \chi_r(f(x)), \]
where $f(x)$ is a polynomial of the form
\[ f(x) = \prod_{j=1}^n (y_j - x)^{k_j},\]
for some $0\leqs k_j < r$. At this point, we use the Weil bound for multiplicative character sums to upper bound these sums point-wise, which delivers the result. This method is uniform in the choice of $Y$ provided $|Y| = n$ is small enough in terms of $q$, which is what enables us to control the quantity $\lbr$. We remark that an averaging argument that is elementary (relative to the Weil bound, at least) can be used to prove the weaker result that $\ul{\alpha} \geqs 1/2$. It seems that breaking the $1/2$-barrier here would require new insights. 

After we proved our theorem, we discovered that the key idea of applying bounds on complete character sums to show pseudorandomness of the quadratic residues is very, very well known; it was already known to Davenport in 1939 \cite{davenportpseudorandomness}, though, of course, he did not have access to the Weil bound via the Riemann hypothesis then, and had to resort to zero-free regions for the associated $L$-functions instead. The reader may peruse \cite{quasirandomZnCG}, \cite{weil1}, \cite{weil2}, \cite{weil3}, \cite{weil4}, \cite{weil5}, \cite{weil6}, \cite{weil7}, \cite{weil8} for a selection of such applications similar to ours, many of which appear to be independent of each other. See also, \cite{MSSSquadratic} for a recent example in the literature which uses similar ideas. 

In Section~\ref{sec: numerical}, we present numerical evidence towards Conjecture~\ref{conj: main squares}. We restrict our attention there to quadratic residues in prime fields, i.e., $q$ prime and $r = 2$. Our programs were written in the Julia language \cite{julialanguage} and are available with their outputs on Dropbox\footnote{\url{https://www.dropbox.com/sh/ov6upgdgq4h692i/AACaE3cC246OW_zD4qTAfi9Ka?dl=0}} and GitHub\footnote{\url{https://github.com/anuragsahay/vcdimsquares}}. A crucial observation of independent interest that underlies our code is the following:
\begin{proposition} \label{prop: invariant}
The property of being shattered by $\Sq$ is translation-dilation invariant. In other words, for any $Y \subseteq \Fq$, $a \in \Fq^\times$ and $b \in \Fq$, 
\[ Y \textnormal{ is shattered by } \Sq \iff aY + b \textnormal{ is shattered by } \Sq, \]
where $aY + b = \{ ay + b : y \in Y \}$. 
\end{proposition}
It is clear that this will save time when computing the VC dimension of $\Sq$. 

We considered three numerical experiments: 
\begin{enumerate}
\item (Section~\ref{exp: bruteforce}) A direct computation of the VC-dimension of $\Sq$ for prime values of $q$ going up to 300.
\item (Section~\ref{exp: arithmetic}) A computation of the size of the largest shattered arithmetic progression in $\Fq$ for prime values of $q$ going up to 200000.
\item (Section~\ref{exp: probabilistic}) A set of plots of the likelihoods that a random set whose size is a given proportion of the theoretical maximum VC-dimension, $\log_2 q$, is shattered for prime values
\[
	58 \approx 2^{5/0.85} \leqs q \leqs 2^{12/0.7} \approx 144716.
\]
\end{enumerate}

While we are not able to compute experiment (1) out to large primes, we do verify that the VC-dimension of $\Sq$ is $\floor{\log_2q} - 1$ for about 57\% of the primes $5 \leqs q \leqs 300$, while the remainder have the maximum theoretical VC-dimension $\floor{\log_2 q}$ (see Figure \ref{fig: vcdim}). In fact, by cutting the search short upon reaching a shattered set of size $\floor{\log_2q} - 1$, we are able to verify the VC-dimension of $\Sq$ is at least $\floor{\log_2 q} - 1$ for primes up to $512$. Perhaps $\vcdim \Sq < \floor{\log_2 q} - 1$ for some $q$, but it is not possible to tell at the small scales at which our experiment runs. If this happens, it likely happens for $q$ just larger than a power of $2$, such as $257$ or $521$. In fact, the search for a shattered set of the desired size takes quite a bit longer for such primes.
\begin{figure}
	\includegraphics[width=0.75\textwidth]{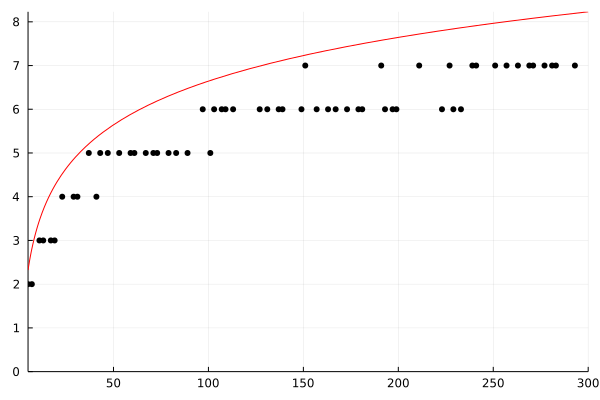}
	\caption{On the horizontal axis: primes $q$ from $5$ to $300$. On the vertical axis: the size of the largest shattered subset found by the first experiment for this $q$. The red curve is the graph of $\log_2 q$, for reference.}
	\label{fig: vcdim}
\end{figure}

For experiment (2), the length of the longest shattered arithmetic progression, in general, varies between $\frac12 \log_2 q$ and $\log_2q$, but certainly seems to concentrate around $\frac{3}{4} \log_2 q$, as indicated in Figure \ref{fig: arith}.
\begin{figure}
	\includegraphics[width=0.75\textwidth]{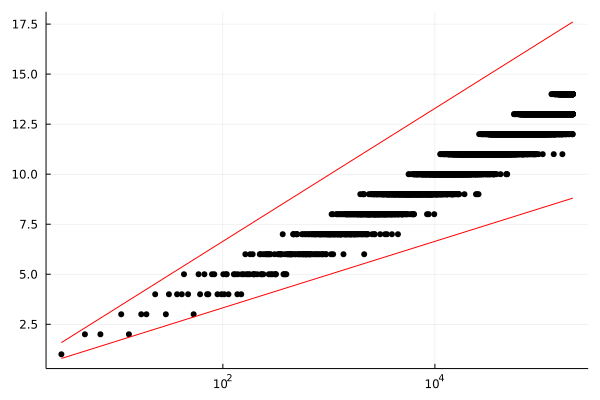}
	\caption{On the horizontal axis: The primes $q$ plotted on a logarithmic scale. Each point has a vertical coordinate equal to the length of the largest shattered arithmetic sequence in $\F_q$. The red upper and lower bounding lines are graphs of $\log_2 q$ and $\frac12 \log_2 q$, respectively.}
	\label{fig: arith}
\end{figure}

The third experiment seems to indicate that, probabilistically speaking, most subsets of size $\floor{0.7 \log_2 q}$ are shattered by $\Sq$, while most subsets of size $\floor{0.85 \log_2 q}$ are not shattered. It is not clear from the data if this remains true for larger primes. (See Figure \ref{Fig: interfaces}.)

Finally, in Section~\ref{sec: openproblems}, we discuss some interesting open problems related to our paper.
\subsubsection*{Notation} For $A$,$B$, and $C$ quantities that depend on various parameters with $C \geqs 0$ always, we use the Vinogradov notation $A - B \ll C$ interchangeably with the Bachmann-Landau notation $A = B + O(C)$ to mean that there exists a constant $K > 0$ such that $|A - B| \leqs KC$. We allow the constant $K$ to depend on some parameters, in which case such dependence is denoted by subscripts such as $A \ll_\epsilon C$ or $A = O_\epsilon(C)$.

\subsection*{Acknowledgements} This article was written while the authors were affiliated with the Department of Mathematics, University of Rochester. The authors are grateful for the pleasant work environment provided by the Department.

BM was partially supported by Alex Iosevich's NSF grant DMS-2154232. AS received summer support from the same grant in 2022. ELW was partially supported by NSF grant DMS-2204397, and by the AMS-Simons travel grants. 

All three authors would like to thank Alex Iosevich for encouragement and useful discussions. AS would like to thank Steve Gonek, Will Sawin, and Igor Shparlinski for useful discussions on character sums. Finally, AS would also like to thank Joe Neeman for useful discussions on VC-dimensions in graphs.

\section{Preliminaries from Fourier analysis and algebraic geometry}

%

We review some basic facts about Fourier analysis on finite Abelian groups which may be found in any relevant textbooks, such as \cite{steinshakarchi} or \cite{terras}. Let $G$ be a finite Abelian group written multiplicatively. A \emph{multiplicative character} on $G$ is a homomorphism $\chi : G \to S^1 \subset \C$, i.e.
\[
	\chi(xy) = \chi(x) \chi(y) \qquad \text{ for all } x,y \in G.
\]
The set $\widehat G$ of multiplicative characters on $G$ forms a group under pointwise multiplication, and is called the \emph{Pontryagin dual} of $G$. In general, the characters satisfy the orthogonality relation
\[
	\frac{1}{|G|} \sum_{\chi \in \widehat G} \chi(x) \overline{\chi(y)} =
	\begin{cases}
		1 & \text{ if $x = y$,} \\
		0 & \text{ otherwise.}
	\end{cases}
\]
Now let $H \subseteq G$ be a subgroup and consider the quotient group $G/H = \{xH : x \in G\}$. The dual $\widehat{G/H}$ of the quotient is naturally identified with the set of characters in $\widehat G$ which are constantly $1$ on $H$. The orthogonality relation above applied to the quotient $G/H$ now reads as
\[ 
	\frac{1}{|G/H|}\sum_{\chi \in \GH} \chi(x)\ov{\chi(t)} = \begin{cases} 1 & \text{ if } x \in tH, \\ 0 &\text{ otherwise.} \end{cases} 
\]
Specializing to $G = \Fq^\times$ and $H = \Sqr$, this reads
\begin{equation*}
\1_{t\Sqr}(x) = \frac{1}{r} \sum_{\chi \in \widehat{\Fq^\times/\Sqr}} \chi(x)\ov{\chi(t)},
\end{equation*}
for $x,t \neq 0$. We now extend the domain of $\chi$ to $0$, by defining $\chi(0) = 0$ for nontrivial characters and $\chi_0(0) = 1$ for the trivial character $\chi_0$. If $q$ is a prime and $r = 2$, this is consistent with the definition of the Legendre symbol, which is partially the motivation for our convention. Further, with the fuzzy notion of set membership encoded in \eqref{fuzzy}, the above relation then holds for all $x \in \Fq, t\in \Fq^\times$. Hopefully, this justifies our non-standard choices in this regard.

As remarked in the introduction, since $\Fq^\times$ is cyclic, $\widehat{\Fq^\times/\Sqr} = \gen{\chi_r}$ where $\chi_r$ is a (not necessarily unique) character that has order $r = |\Fq^\times/\Sqr|$ and hence generates the dual group. Thus, one can write the above as 
\begin{equation} \label{orthogonality}
\1_{t\Sqr}(x) = \frac{1}{r} \bigg(1+\sum_{k=1}^{r-1} \chi_r^k(x)\ov{\chi_r^k(t)}\bigg),
\end{equation}
where we have separated the trivial character. 

We will use (as a black box) a deep result from algebraic geometry, namely the Weil bound for character sums.
\begin{lemma}[Weil] \label{lem: weil}
For $r\geqs 2$, suppose that $f \in \Fq[x]$ has $n$ distinct roots and that $f$ is not an $r$th power. Then, we have
\[ \left|\sum_{x\in \Fq} \chi_r(f(x)) \right| \leqs (n-1)\sqrt{q}. \]
\end{lemma}
We refer the reader to \cite[Chapter~11]{iwanieckowalski} for an account of this result which assumes minimal knowledge of algebraic geometry. In particular, our lemma as stated is a special case of \cite[Theorem~11.23]{iwanieckowalski}. 

We apply the Weil bound to prove the following lemma that attests to the pseudorandomness of bounded index subgroups $\Sqr \subseteq \Fq^\times$.
 
\begin{lemma} \label{lem: main}
Let $r \geqs 2$, $n\geqs 1$, $Y = \{y_1,\cdots,y_n\} \subseteq \Fq$ and $\bt = (t_1,\cdots,t_n) \in (\Fq^\times)^n$. Then,
\[ \Pr_{x \in \Fq}\Big[y_j - x \in t_j\Sqr, 1\leqs j \leqs n\Big] = \frac{1}{r^n} + O\Big(\frac{n}{q^{1/2}}\Big), \]
where the implicit constant can be chosen to be $1$ and hence is uniform in all parameters, and $x$ is sampled from the uniform distribution. \end{lemma}
Note that for a fixed $j, y_j, t_j$, 
\[ \Pr_{x\in\Fq}\Big[y_j - x \in t_j\Sqr\Big] = \frac{1}{r}, \] 
since $y_j - x$ is also uniform in $\Fq$, and a coset of $\Sqr$ has measure $1/r$. Thus, the above lemma can be viewed as saying that for a fixed choice of $n$, the values of $x - y_j$ all equidistribute in the cosets of $\Sqr$ independently of each other as $q \to \infty$. This is what we mean by the pseudorandomness of bounded index subgroups.

\begin{proof}[Proof of Lemma~\ref{lem: main}]
For conciseness, let $E \subseteq \Fq$ be the event in the lemma (namely, that $y_j - x \in t_j \Sqr$ for every $1\leqs j \leqs n$). Further, for $1\leqs j \leqs n$, define $E_j \subseteq \Fq$ to be the event that $y_j - x \in t_j \Sqr$. Then, clearly $E = \cap_j E_j$, and hence
\[ \1_E(x) = \prod_{j=1}^n \1_{E_j}(x) = \prod_{j=1}^n \1_{t_j\Sqr}(y_j - x). \]

We now use \eqref{orthogonality} to Fourier-expand each indicator function in the right-most expression, obtaining

\[ \frac{1}{r^n} \prod_{j=1}^n \bigg(1 + \sum_{k=1}^{r-1} \chi_r^k(y_j-x)\ov{\chi_r^k(t_j)}\bigg).\]
Expanding out the product, we find that a convenient index for the resulting summand is $\bk = (k_1,\cdots,k_n)$, where $k_j$ is the choice of $k$ for a given $j$ taken from the inner sum. Thus,
\[ \1_E(x) = \frac{1}{r^n} \bigg(1 +\sum_{\substack{0\leqs \bk \leqs r-1\\\bk \neq \bz}}\prod_{j=1}^n \chi_r^{k_j}(y_j - x)\ov{\chi_r^{k_j}(t_j)}\bigg), \]
where $0 \leqs \bk \leqs r-1$ means that $0\leqs k_j \leqs r-1$ for each $1\leqs j \leqs n$, and $\bk \neq \bz$ means that $k_j$ is not identically $0$. 

Recall that $\chi_r$ is a multiplicative character, and hence that $\chi_r^k(t) = \chi_r(t^k)$, and further that $\chi_r(x)\chi_r(x') = \chi_r(xx')$. In particular, if we define
\[ b(\bk,\bt) = \ov{\chi_r}\bigg(\prod_{j=1}^n t_j^{k_j}\bigg),\] 
and
\[ f_\bk(x) = \prod_{j=1}^n (y_j - x)^{k_j}, \]
then the above can be written as
\[ \1_E(x) = \frac{1}{r^n}\bigg(1 + \sum_{\substack{0\leqs \bk \leqs r-1\\\bk \neq \bz}} b(\bk,\bt)\chi_r(f_\bk(x))\bigg).\]
Now, we find by linearity of expectation that
\[ \Pr_{x \in \Fq}[E] = \E_{x} [\1_{E}(x)] = \frac{1}{r^n} + \frac{1}{r^n} \sum_{\substack{0\leqs \bk \leqs r-1\\\bk \neq \bz}} b(\bk,\bt)\E_{x} [\chi_r(f_\bk(x))]. \]
Here the expectation is over $x \in \Fq$. Now, $|b(\bk,\bt)|\leqs 1$. Thus, by the triangle inequality, 
\begin{equation*}\begin{split} 
\bigg(\Pr_{x \in \Fq}[E] - \frac{1}{r^n}\bigg) 
& \leqs \frac{1}{r^n}\sum_{\substack{0\leqs \bk \leqs r-1\\\bk \neq \bz}} \bigg|\E_{x} [\chi_r(f_\bk(x))]\bigg| 
\\
& \leqs \max_{\substack{0\leqs \bk \leqs r-1\\\bk \neq \bz}} \bigg|\E_{x} [\chi_r(f_\bk(x))]\bigg|.
\end{split} \end{equation*}

It now suffices to show that uniformly for $0 \leqs \bk \leqs r-1, \bk \neq \bz$, we have the bound
\[ \E_{x} [\chi_r(f_\bk(x))] \leqs \frac{n}{q^{1/2}}.\]
This follows by multiplying the Weil bound (Lemma~\ref{lem: weil}) by $1/q$ on both sides. To see this, note that $f_\bk \in \Fq[x]$ is clearly not an $r$th power, since $k_j < r$ for each $j$, and further note that $f_\bk$ has at most $n$ distinct roots, as its roots are contained in $Y = \{y_1,\cdots,y_n\}$. This completes the proof.
\end{proof}

\section{Proof of Theorem~\ref{thm: main gen}}
Let $\epsilon > 0$, and 
\begin{equation} \label{m bound} n \leqs \Big(\frac{1}{2} - \epsilon\Big) \log_r q. \end{equation}
We will first show that for $q$ large enough in terms of $r$ and $\epsilon$, this will imply that
\begin{equation} \label{probability > 0} \Pr_{x \in \Fq}\Big[y_j - x \in t_j\Sqr, 1\leqs j \leqs n\Big] > 0. \end{equation}
uniformly in $Y = \{y_1,\cdots,y_n\}$ and $\bt \in (\Fq^\times)^n$. By Lemma~\ref{lem: main}, we see that the probability above is at least
\[ \frac{1}{r^n} - \frac{n}{q^{1/2}} \geqs \frac{1}{q^{1/2 - \epsilon}} - \frac{\log_r q}{q^{1/2}}. \]
Since $r \ll 1$, $\log_r q = o(q^{2\epsilon})$ as $q \to \infty$. Thus,
\[ \Pr_{x \in \Fq}\Big[y_j - x \in t_j\Sqr, 1\leqs j \leqs n\Big] \gg q^{-1/2+\epsilon} > 0, \]
as claimed.

To prove the theorem, it clearly suffices to show that for any set of size $n$ satisfying \eqref{m bound}, $Y$ is shattered by $\Sqr$. Let $Y = \{y_1,\cdots,y_n\}$. To prove that $Y$ is shattered by $\Sqr$, we need to find, for every $\emptyset \subseteq A \subseteq Y$, an element $x \in \Fq$ with the property that
\[
	A = Y \cap (\Sqr + x).
\]

We do this as follows. Fix $t \in \Fq^\times \setminus \Sqr$. Now, define $\bt = (t_1,\cdots,t_n)$ by
\[ t_j = \begin{cases} 1 & \text{ if } y_j \in A, \\ t & \text{ if } y_j \notin A. \end{cases} \]

Applying \eqref{probability > 0} to this choice of $\bt$, we find that, since this probability is non-zero, there must be an $x \in \Fq$ with the property that $y_j - x \in t_j\Sqr$ for each $1\leqs j\leqs n$. In particular, since $t \notin \Sqr$, this means that $y_j - x \notin \Sqr$ when $t_j = t$ and $y_i - x \in \Sqr$ when $t_j = 1$. Thus, $y_j \in (\Sqr + x)$ if and only if $t_j = 1$, which by construction, happens if and only if $y_j \in A$. Since this is true for any choice of $A$, we have shown that $Y$ is shattered by $\Sqr$, thus obtaining Theorem~\ref{thm: main gen}.

\section{Numerical evidence for Conjecture~\ref{conj: main squares}} \label{sec: numerical}
First, we prove Proposition~\ref{prop: invariant}. Some care must be taken to resolve issues surrounding whether $0 \in \Sqr$ -- for example, by recalling the modified definition of shattering proposed in the introduction. 
\begin{proof}[Proof of Proposition~\ref{prop: invariant}]
The invariance under translations is trivial, since the set of translates of $\Sq$ is obviously invariant under translations. A similar argument shows that if $a \in \Sq$, then $aY$ is shattered by $\Sq$ if and only if $Y$ is shattered by $\Sq$. 

It remains to show that if $a \in \Fq^\times \setminus \Sq$, the $aY$ is shattered by $\Sq$ if and only if $Y$ is shattered by $\Sq$.

We suppose $Y = \{y_1,\cdots,y_n\}$ and let $\bt \in \{0,1\}^n$ be a bit-string. We say that $Y$ realizes $\bt$ if there is a translate $x\in \Fq$ so that
\[ y_j \in (\Sq + x) \iff t_j = 1. \]
We see that an equivalent way to saying that $Y$ is shattered by $\Sq$ is that $Y$ realizes every $\bt \in \{0,1\}^n$. Let $\bt' = (1,1,\cdots,1) - \bt$. Then we see that
\begin{equation*} Y \text{ realizes } \bt \iff aY \text{ realizes } \bt', \end{equation*}
since if $x\notin Y$ is the translation such that $\bt = \1_{Y \cap (\Sq + x)}$, then $ax \notin aY$ is a translation for which $\bt' = \1_{a Y \cap (\Sq + ax)}$. 
If $Y$ is shattered by $\Sq$, then every choice of $\bt \in \{0,1\}^n$ is realized by $Y$. But $\bt'$ runs over all of $\{0,1\}^n$ in that case, thus establishing that $aY$ is shattered by $\Sq$, as desired.
\end{proof}

We now discuss the three algorithms described in the introduction in order. Preceding this, however, we discuss our code that detects if a given subset is shattered $\Sq$ that is common to all algorithms.

\subsection{Detecting if a subset is shattered} \label{exp: basic}

Common to all code is a way to determine if a subset $Y \subset \mathbb F_q$ of size $n$ is shattered by the shifts of the squares $\Sq$. This is implemented in such a way to replace loops over the elements of $\F_q$ with vectorized operations. To implement this, we define a $q \times n$ matrix $A$ whose entries are $1$'s and $0$'s such that
\[
	A(x,i) = \begin{cases}
		1 & \text{ if } y_i - x \in \Sq \\
		0 & \text{ otherwise}
	\end{cases}
	= \1_{y_i - \Sq}(x) = \1_{x + \Sq}(y_i).
\]
Note, the $i$th column of $A$ reads as the indicator function of the reflected set of squares $-\Sq$ after translating by $y_i$, and the $x$th row of $A$ reads as the restriction of the shifted set of squares $x + \Sq$ to the set $Y = \{y_1,\ldots, y_n\}$. Hence, if we are to initialize such a matrix in Julia, we would first initialize the reflected set of squares modulo $q$, and then set the columns of $A$ as follows.
\begin{verbatim}
    using CircularArrays
    
    # ...

    nSq = CircularArray(zeroes(Int64, q))
    for x in 0:(q-1)/2
        nSq[-x^2] = 1
    end

    # ...

    A = Array{Int64}(undef, q, n)
    for i in 1:n
        A[:,i] = nSq[Y[i]:Y[i]+q-1]
    end
\end{verbatim}
The set $Y$ is shattered by shifts of $\Sq$ if and only if every length-$n$ string of $1$'s and $0$'s is realized as a row in $A$. To detect this, we read a row as if it were a number written in binary, except we read the digits in reverse order so that the first element is the $1$'s digit, the second the $2$'s digit, the third the $4$'s digit, and so on (the reversed order has some practical benefit for the algorithm in Section~\ref{exp: arithmetic}). We can carry this operation out for all rows with a single matrix multiplication:
\[
	A \begin{bmatrix}
		1 \\
		2 \\ 
		4 \\
		\vdots \\
		2^{n-1}
	\end{bmatrix}.
\]
The set $Y$ is shattered then if every number from $0$ to $2^n - 1$ appears as an entry in the resulting vector. We can get Julia to return ``true" if $Y$ is shattered and ``false" if $Y$ is not shattered by writing:
\begin{verbatim}
    restrictions = CircularArray(falses(2^n))
    for k in A * [2^i for i in 0:n-1]
        restrictions[k] = true
    end
    
    return all(restrictions)
\end{verbatim}

\subsection{Computing the VC-dimension} \label{exp: bruteforce} This experiment computes the VC-dimension of $\Sq$ for primes $q$ in a user-defined range. Broadly speaking, this is accomplished by performing an exhaustive search through the subsets of $\F_q$ for the largest shattered by translates of $\Sq$. The search is slightly smarter than brute-force, and is described below.

The code finds the VC-dimension of $\Sq$ for primes $q$ which are $5$ or greater. One quickly sees that the VC-dimension of $\Sq$ for such $q$ is at least $2$-- $\{0,1\}$ is always shattered. Due to Proposition~\ref{prop: invariant}, it is enough to search over only those subsets which contain $0$ and $1$, as one can always find an appropriate affine transformation.

The set $\{0,1\}$ will serve as the root of the tree of subsets over which we will search. A set $Y$ on the tree will have parent $Y \setminus \max Y$ after identifying $\F_q$ with the integers $\{0,1,\ldots, q-1\}$. Note, a set $Y$ is the ancestor to exactly $q - 1 - \max Y$ generations. To search for the largest shattered set, we start with the following iterative scheme which we will refine momentarily:
\begin{enumerate}
\item Determine if $Y$ is shattered. If so, record the size of $Y$ if it is the largest so far. If $Y$ is shattered and has a child, go to step (2). Otherwise, go to step (3).
\item Take the first child $Y \cup \{\max Y + 1\}$ of $Y$ and assign it the name $Y$. Go to step (1).
\item If $Y$ has a next sibling, namely $Y$ with $\max Y$ removed and $\max Y + 1$ included, then assign it the name $Y$ and go to step (1). Otherwise, go to step (4).
\item If $Y$ is not the root $\{0,1\}$, take its parent $Y \setminus \max Y$ and assign it the name $Y$ and go to step (3). If $Y$ is the root, exit.
\end{enumerate}
This iterative scheme certainly gets the job done, but it can be improved in two ways at relatively little expense.

First, we need not search through the descendants of $Y$ if there are none whose size exceeds the size of the largest shattered set found so far. This is an easy check, since the largest descendant of $Y$ has size $|Y| + q - 1 - \max Y$, the size of $Y$ plus the number of generations following it.

The second refinement uses what we call the \emph{shattering index} of $Y$ in the code. This is determined by first tallying the number of times each restriction is realized, taking the minimum of these numbers, and then taking the base-2 logarithm of that. We can get Julia to return the shattering index of $Y$ by setting up the matrix $A$ as before and writing:
\begin{verbatim}
    restrictions = CircularArray(zeros(2^n))
    for k in A * [2^i for i in 0:n-1]
        restrictions[k] += 1
    end
    
    min = minimum(restrictions)
    
    if min > 0
        return floor(Int64, log2(min))
    else
        return -1
    end
\end{verbatim}
The $-1$ value here is used as a flag to indicate that our set is not shattered. If $Y$ has shattering index $k$, then any superset of size $|Y| + k + 1$ necessarily fails to be shattered. Hence, we should only search through the descendants of sets $Y$ for which $|Y| + k$ exceeds the size of the largest shattered set found so far.

\subsection{Size of the largest shattered arithmetic sequence} \label{exp: arithmetic} In this experiment, we compute the largest value $n \leqs \log_2 q$ such that $Y = \{0,1,\cdots,n-1\} \subseteq \Fq$ is shattered by $\Sq$ for primes $q$ in a user-defined range. 

The motivation for this experiment comes from a naive guess that perhaps a good candidate for a large shattered subset of $\Fq$ for $q$ prime is a long arithmetic progression. At this point, we see that from Proposition~\ref{prop: invariant} it follows that if an arithmetic progression of length $n$ is shattered, then so is the set $\{0,1,\cdots, n-1\}$.

To find the largest $n$ for which $Y = \{0,1,\cdots, n-1\}$ is shattered by $\Sq$, the code proceeds by as described in Section~\ref{exp: basic} with $n = \lfloor \log_2 q\rfloor$ (i.e., with the theoretical maximum value it could be). The main time-saving innovation here is that it suffices to do this computation only for this value of $n$. This is achieved by the observation that if $\{0,1,\cdots,n-1\}$ is not shattered, then the restrictions matrix computed when checking already contains the information about whether $\{0,1,\cdots,n-2\}$ is shattered or not. Let $R$ be this restriction matrix, which has dimensions $2^n \times 1$. Let $R_1$ and $R_2$ be the $2^{n-1} \times 1$ matrices such that
\[ R = \begin{bmatrix} R_1 \\ R_2 \end{bmatrix}. \]
Now, let $R'$ be the $2^{n-1} \times 1$ matrix obtained by applying the component-wise logical-OR operation to $R_1$ and $R_2$. Then $R'$ is the restriction matrix that is obtained when determining if $\{0,1,\cdots,n-2\}$ is shattered. 

This is best illustrated by an example. Suppose $n = 6$. Then, the $3$rd element of $R$ represents whether the bit-string $110000$ can be achieved as the indicator function of $\{0,\cdots,5\} \cap (\Sq + x)$ for some $x$ as $1 + 2 = 3$. Similarly, the $19$th element of $R$ represents whether the bit-string $110001$ can be achieved as $1 + 2 + 16 = 19$.

Thus, if we take the logical-OR of these two Boolean values, we get a bit which represents whether $11000$ can be achieved as the indicator function for $\{0,\cdots,4\} \cap (\Sq + x)$ for some $x$, as we are essentially neglecting to care what happens to $5$. In particular, this will be stored in the $3$rd position in $R'$, which lines up with the fact that $1 + 2 = 3$. 

We see that this process is recursive, and thus, once we have computed the restriction matrix for $n = \lfloor \log_2 q\rfloor$, we can easily find the actual largest value of $n$ that works using a simple loop, given below:

\begin{verbatim}
    while !all(restrictions[0:2^n-1])
        restrictions[0:2^(n-1)-1] 
            = restrictions[0:2^(m-1) - 1] .| 
              restrictions[2^(n-1) : 2^n - 1]
        n -= 1
    end
    
    return n
\end{verbatim}

\subsection{Probabilities that a set is shattered} \label{exp: probabilistic}

	\begin{figure}
	\includegraphics[width=0.45\textwidth]{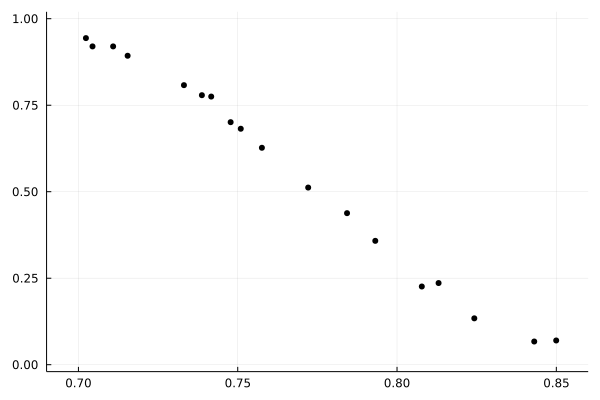} 
	\includegraphics[width=0.45\textwidth]{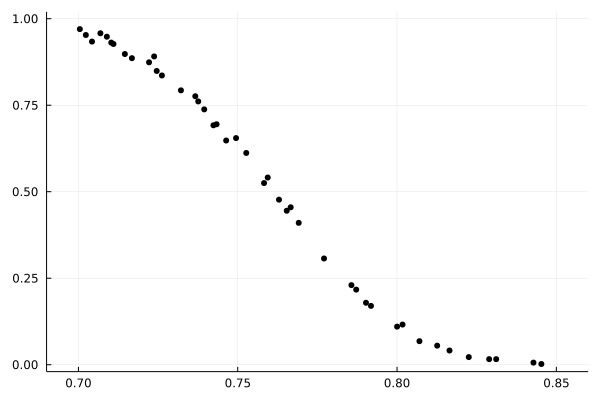} \\
	\includegraphics[width=0.45\textwidth]{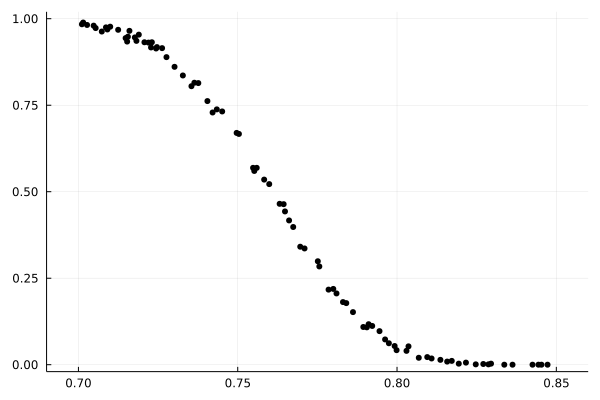} 
	\includegraphics[width=0.45\textwidth]{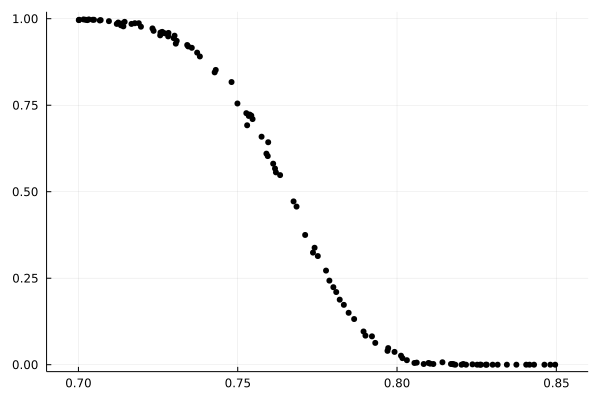} \\
	\includegraphics[width=0.45\textwidth]{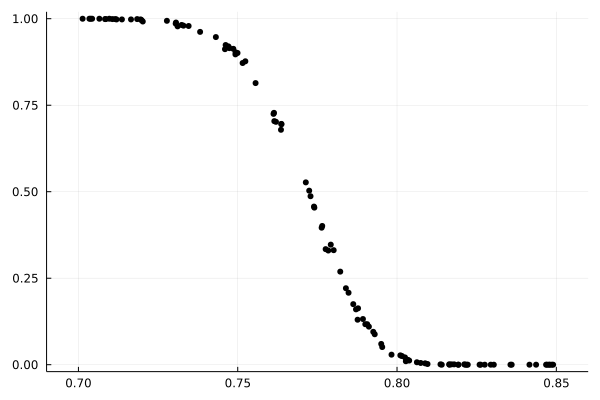} 
	\includegraphics[width=0.45\textwidth]{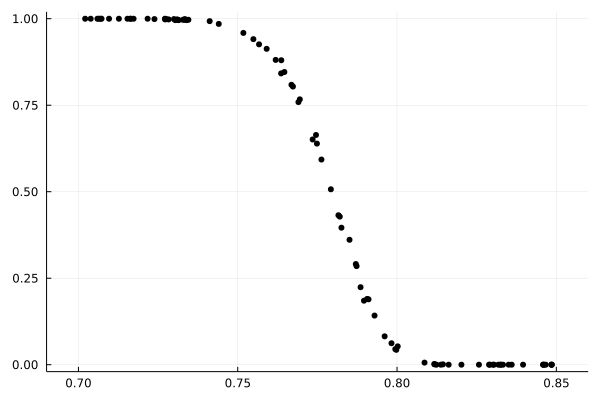} \\
	\includegraphics[width=0.45\textwidth]{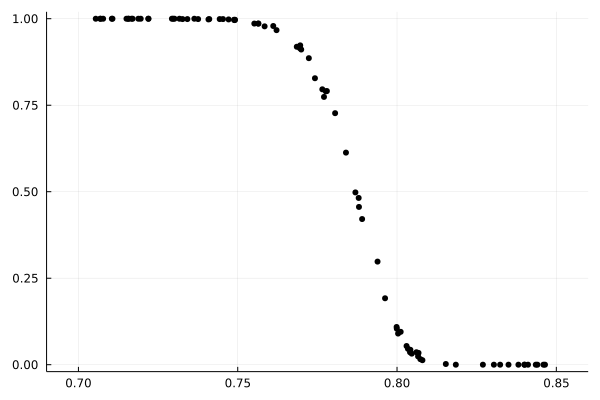} 
	\includegraphics[width=0.45\textwidth]{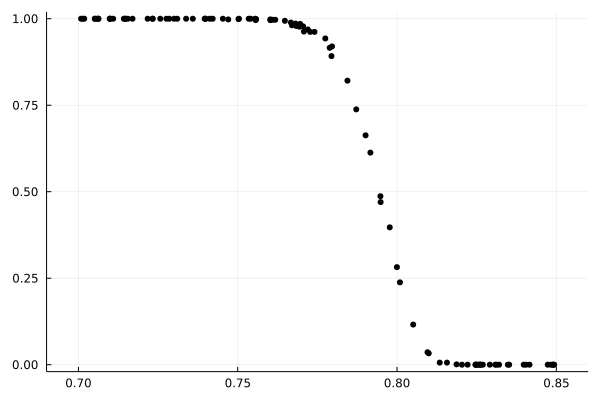} \\
	\caption{Along the vertical axis is probability. Along the horizontal axis is the proportion $n/\log_2(q)$ and ranges in $[0.7, 0.85]$. From left to right starting at the top, the plots take $n$ to be at $5,\ldots,12$, respectively, and the series in each plot is generated as $q$ ranges such that $0.7 \leqs n/\log_2q \leqs 0.85$. Primes $q$ were selected randomly so that the expected number of values of $q$ per plot is about $100$.}
	\label{Fig: interfaces}
	\end{figure}

This experiment determines, up to statistical error, the likelihood that a random subset of $\F_q$ of size $n$ is shattered by $\Sq$. In particular, for each $n$ in a user-defined range, we draw $1000$ subsets of size $n$ from $\F_q$ (pseudo-)randomly and uniformly and determine the proportion of them which are shattered. We then record the datum
\[
	\left(\frac{n}{\log_2 q} , \frac{\# \text{ shattered}}{1000} \right).
\]
We generate a plot of probabilities by fixing $n$ and ranging over primes $q$ in $[2^{n/0.85}, 2^{n/0.7}]$ (see Figure \ref{Fig: interfaces}).

Plotting the recorded data on the axes reveals a smooth-looking profile of shattering probabilities which seems to be monotone decreasing from nearly $1$ to nearly $0$ on the vertical axis as we range from $0.7$ to $0.85$ on the horizontal axis. There also seems to be an interface that simultaneously scales thinner and shift rightwards as $n$ increases. It would be interesting to investigate (1) the position of the interface as $n \to \infty$, (2) the scale of the interface as $n \to \infty$, and (3) the structure of the interface after an appropriate rescaling. We anticipate that this would require a combination of techniques from both algebraic geometry and probability. It would follow that nearly every subset in a relatively large range of sizes is shattered by $\Sq$, rather than just those close to $\frac 12 \log_2 q$.

\section{Directions for future research} \label{sec: openproblems}

The following are a list of open problems related to our results. We would be interested in a solution for any of these.

\subsection{Sets which are not shattered}
For any $\delta>0$, is it always the case that for $q$ large enough, there is a set $Y = \{y_1,\cdots,y_n\}$ with $n = |Y| \geqs (\tfrac12+\delta) \log_2 q$ such that $\Sq$ does not shatter? It seems plausible that this is true, which would imply, together with Theorem~\ref{thm: main squares} that $\ul{\beta} = 1/2$. It is unclear whether the right approach to such an endeavour would be algebraic geometric or probabilistic. 

The obvious analogue for $\Sqr$ is also interesting, and in particular, it seems unclear what the right value for $\lbr$ must be. A concrete question of interest to us is where $\lbr \to 0$ as $r \to \infty$; we remark that our lower bound in Theorem~\ref{thm: main gen} has this property, but it is not clear if the true value of $\lbr$ satisfies this as well.

\subsection{The probabilistic interface}

For a fixed $q$ and $n$, let $p_{q,n}$ denote the probability that a random $Y \subseteq \F_q$, $|Y| = n$ selected with uniform probability is shattered by $\Sq$. The experiments suggest there may be an interface over which $p_{q,n}$ goes from essentially $1$ to essentially $0$. To put it precisely, let
\begin{align*}
	\ul\gamma &= \sup \left\{ t \in [0,1] : \liminf_{q \to \infty} \inf_{n \leqs t \log_2 q} p_{q,n} = 1 \right\}, \\
	\ol \gamma &= \inf \left\{ t \in [0,1] : \limsup_{q \to \infty} \sup_{n \geqs t \log_2 q} p_{q,n} = 0 \right\}.
\end{align*}
It is difficult to ascertain the behavior of $\underline \gamma$ and $\overline \gamma$; we can only say for certain that
\[
	\ul{\beta} \leqs \ul{\gamma} \leqs\ol{\gamma} \leqs\ol{\alpha}.
\]
It seems reasonable to conjecture that $p_{q,n}$ is monotone-decreasing with $n$, from which we would be able to conclude
\[
	\ol{\beta} \leqs\ul{\gamma} \leqs\ol{\gamma} \leqs\ul{\alpha}.
\]
It also seems likely from the experiment in Section \ref{exp: probabilistic} that the scale of the interface tends to $0$ as $q \to \infty$, that is, $\ul{\gamma} = \ol{\gamma}$. It would be interesting to understand the scale, position, and structure of this interface as $q \to \infty$.

\subsection{VC-dimension of random graphs}

Since $\Pq$ is a pseudorandom graph, it is a deterministic approximation to an Erd\H{o}s-R\'enyi random graph $\ER(n,p)$ with $n = q$ and $p = 1/2$. This raises the question of what the VC-dimension of an Erd\H{o}s-R\'enyi random graph behaves like. This problem has been studied in the sparse setting as $n \to \infty$ with $p = p(n) = o(1)$ by Anthony, Brightwell, and Cooper \cite{VCdimrandomgraph}. However, there appear to be no results in the literature in the dense regime where $p$ is constant as $n \to \infty$.

An algebraic variant of the above which is perhaps of more relevance is the problem of determining the VC-dimension of a random Cayley (di)graph, $\RC(G,p)$. Here, for a fixed group $G$, we pick a set $S$ where independently for each $x \in G$, we have $x \in S$ with probability $0\leqs p\leqs 1$. Then, the graph we consider is $\Cay(G,S)$. 

One could also use other random graph models instead of $\ER(n,p)$ and $\RC(G,p)$ -- we would be interested in any such results, as it may shed light on whether Conjecture~\ref{conj: main squares} is correct.

\bibliography{vcdim}
\bibliographystyle{plain}

\end{document}